\documentclass[11pt,twoside]{article}
\usepackage{amssymb,epsfig}
\usepackage{mathrsfs}
\usepackage{amsthm}
\usepackage{amsmath}
\usepackage{graphicx}
\usepackage{multirow}
\usepackage{array}
\usepackage{tabularray}
\usepackage{enumerate}
\usepackage{paralist}
\usepackage{setspace}
\usepackage{xcolor}
\usepackage{mathtools}
\usepackage{subcaption}

\date{}
\makeatletter
\newcommand*\bigcdot{\mathpalette\bigcdot@{.5}}
\newcommand*\bigcdot@[2]{\mathbin{\vcenter{\hbox{\scalebox{#2}{$\m@th#1\bullet$}}}}}
\makeatother
\newtheorem{proposition}{Proposition}[section]
\newtheorem{theorem}[proposition]{Theorem}
\newtheorem{lemma}[proposition]{Lemma}
\newtheorem{corollary}[proposition]{Corollary}
 \newtheorem{remark}[proposition]{Remark}

\begin{document}

\title{Bistable Dynamics and Complete Classifications of A Seasonal Competition Model Between {\it Aedes aegypti} and {\it Aedes albopictus}
\thanks{
\textbf{Funding:} {Niu's research was supported by NSFC (grant 12671210) and the Natural Science Foundation of Shanghai (grant
24ZR1401900). Wang's research was supported by the Strategic Priority Research Program of Chinese Academy of Sciences (XDB0900100), NSFC (grant 12331006), and National Key R\&D Program of China (2024YFA1013603, 2024YFA1013600). Xie's research was supported by NSFC (grant 12671195) and the Natural Science Foundation of Fujian Province (grant 2026J001752).
Zheng's research was supported by NSFC (grant 12371484).}}
}
\setlength{\baselineskip}{18pt}
\author{Lei Niu
\thanks{School of Mathematics and Statistics, Donghua University, Shanghai
201620, China (lei.niu@dhu.edu.cn).}
, Yi Wang
\thanks{School of Mathematical Sciences, University of Science and Technology of China,
Hefei, Anhui 230026, China
  (wangyi@ustc.edu.cn).}
  , Xizhuang Xie
  \thanks{School of Mathematical Sciences, Huaqiao University,
Quanzhou, Fujian 362021, China
(xzx@hqu.edu.cn).},
 and Bo Zheng
 \thanks{Corresponding author. School of Mathematics and Information Science, Guangzhou University,
Guangzhou, Guangdong 510006, China (zhengbo611@outlook.com).}
\\[5pt]
}

\maketitle
\begin{abstract}
\begin{spacing}{1.2}
\vspace{3mm} {The bistable dynamics and a complete classification of global dynamics are established for a competition model between {\it Aedes aegypti} and {\it Aedes albopictus} subject to favorable –unfavorable seasonal shifts, originally from the problem proposed in Zheng et al. [SIAM J. Appl. Math. 85 (2025), 1189–1211]. More precisely, by employing Floquet theory, we prove that the unique positive periodic solution must be unstable. Consequently, bistable dynamics naturally emerge, which provides a rigorous theoretical validation of the numerical simulation results reported in the aforementioned work. Building on this, we further derive a full classification of the model's global dynamics. In addition, we carry out a comprehensive analysis of how seasonal switching mechanisms govern the competitive outcomes between these two mosquito species, thereby offering insights into their ecological coexistence and potential displacement.

}
\end{spacing}
	\par
	\textbf{Keywords}: {\it Aedes aegypti}, {\it Aedes albopictus}, Poincar\'{e} map, Bistable dynamics, Floquet multipliers
    \vskip 3mm
    \par
	\textbf{AMS Subject Classification (2020)}: 34C25, 37C65, 92D25
\end{abstract}

\section{Introduction}\label{sec1}
{\it Aedes aegypti} and {\it Aedes albopictus} rank among the world's most prolific mosquito vectors, having spread to more than a hundred countries \cite{lwande2020,Wilkerson}. Over the past century, both species have served as major vectors for numerous mosquito-borne diseases, most notably yellow fever (YFV), dengue (DENV), chikungunya (CHIKV), and Zika (ZIKV) viruses \cite{Bhatt2013,Girard2020,Naslund2021}. Climate change is altering their geographical ranges: as formerly unsuitable habitats become climatically favorable, the distributions of both mosquitoes are projected to shift substantially in the coming century \cite{Gabriel2023,Iwamura,Khormi}. These expanding ranges are expected to pose serious challenges to global public health \cite{Diagne,Mayer}.

{\it Aedes aegypti} is strongly anthropophilic and tends to be found predominantly in urban areas. In contrast, {\it Aedes albopictus} occupies a wider range of habitats, including peri-urban, vegetated, and rural zones, largely due to its broader host preferences and greater adaptability to varying climates. In recent decades, both species have gradually expanded from tropical and subtropical regions into temperate zones.

{\it Aedes albopictus} is particularly invasive and geographically widespread, often coexisting with or displacing {\it Aedes aegypti} depending on environmental conditions (see, for example, \cite{Parker2019,Ruan2026}). Furthermore, studies done by Nasci et al. \cite{Nasci1989} provide compelling evidence that
male {\it Aedes albopictus} can successfully mate with female {\it Aedes aegypti}, yet these crosses typically yield no viable offspring, creating a reproductive disadvantage for {\it Aedes aegypti},
whereas male {\it Aedes aegypti} rarely mate with female {\it Aedes albopictus}. Such reproductive interference significantly influences the population dynamics of both mosquito species and, consequently, indirectly alters the range and intensity of mosquito-borne disease transmission \cite{Yang}.


Additionally, the seasonal migration of species represents a significant and well-recognized biological phenomenon. This process, referred to as seasonal succession by Sommer et al. \cite{SUZLD1986} and Reynolds \cite{Reynolds1980}, or as seasonal alternation by Koch \cite{Koch1974}, illustrates how periodic environmental changes influence numerous ecosystem processes and, consequently, the dynamics of populations and communities. Extensive research has been dedicated to this subject; for further reading, we direct readers to \cite{SX2012, K2010, NWX2025, NWX2026, Tang2025, White2020, Xiao2026} and the references therein.

Driven by these seasonal shifts, {\it Aedes aegypti} and {\it Aedes albopictus} have evolved a diapause mechanism, that is a form of dormancy enabling their eggs to survive unfavorable conditions. This phenomenon has been widely documented in both mosquito species \cite{Barbara, Denlinger,Sylvia}. Specifically, it has been established that the persistence of {\it Aedes aegypti} in temperate regions relies on egg survival during the winter-spring transition, which ensures population continuity into the subsequent favorable season (see, for example, \cite{Girard2020,Gabriel2023,K2010}).


To capture the interactive dynamics between {\it Aedes aegypti} and {\it Aedes albopictus} under interspecific mating competition and diapause induced by seasonal shifts, Zheng et al. \cite{Zheng2025} recently formulated the following time-switched competition model for these two mosquito populations:
\begin{equation}\label{AA}
	\begin{cases}
		\cfrac{dx}{dt} = \cfrac{\beta_1 x^2(t)}{x(t) + \alpha y(t)} - \left[\mu_1 + \xi_1(x(t) + y(t))\right] x(t), \quad  t \in [nT, nT+\bar{T}], \\
		x((n+1)T) = \delta_1 x(nT+\bar{T}), \quad\quad\quad\quad\quad\quad \quad \ t\in [nT+\bar{T},(n+1)T]\\
		\cfrac{dy}{dt} = \beta_2 y(t) - \left[\mu_2 + \xi_2(x(t) + y(t))\right] y(t),\quad \quad \ t \in [nT, nT+\bar{T}], \\
		y((n+1)T) = \delta_2 y(nT+\bar{T}), \quad\quad\quad\quad\quad\quad \quad \ \ t\in [nT+\bar{T},(n+1)T],
	\end{cases}
\end{equation}
 where $x(t)$ and $y(t)$ respectively represent the population size of {\it Aedes aegypti} and {\it Aedes albopictus}, $T$ represents the total length of a seasonal cycle, $\bar{T}\in(0, T)$ denotes the proportion of the cycle during which growth and competition occur, and $n\in\mathbb{Z}_{+}:=\{0, 1, 2, \cdots\}$ indexes the season. The parameters $\beta_{1}$ and $\beta_{2}$ are the birth rates, $\mu_{1}$ and $\mu_{2}$ are the density-independent death rates, $\xi_{1}$ and $\xi_{2}$ are the density-dependent death rates, $\alpha$ measures the mating competitiveness of \textit{{\it Aedes albopictus}} males relative to \textit{{\it Aedes aegypti}} males, and $\delta_{1}$ and $\delta_{2}$ represent the overwintering survival capacities of \textit{{\it Aedes aegypti}} and \textit{{\it Aedes albopictus}}, respectively. From a biological standpoint, all parameters $\alpha$, $T$, $\beta_{i}$, $\mu_{i}$, $\xi_{i}$, $\delta_{i}$ $(i=1,2)$ are positive constants, with $\mu_{i}<\beta_{i}$ and $\delta_{i}\leq 1$.

Clearly, during the favorable season, i.e., for $t\in [nT,nT+\bar{T}]$, the dynamics of the two mosquito species are governed by the continuous growth and competition system
\begin{equation}\label{AA1}
	\begin{cases}
		\displaystyle \cfrac{dx}{dt} = \frac{\beta_1 x^2(t)}{x(t) + \alpha y(t)} - \bigl[\mu_1 + \xi_1\bigl(x(t) + y(t)\bigr)\bigr] x(t), \\[10pt]
		\displaystyle \cfrac{dy}{dt} = \beta_2 y(t) - \bigl[\mu_2 + \xi_2\bigl(x(t) + y(t)\bigr)\bigr] y(t).
	\end{cases}
\end{equation}
Conversely, during the unfavorable season, i.e., for $t\in [nT+\bar{T},(n+1)T]$, the populations enter a diapause state described by the discrete survival mapping
\begin{equation}\label{AA2}
	\begin{cases}
		x\bigl((n+1)T\bigr) = \delta_1 x\bigl(nT+\bar{T}\bigr), \\[4pt]
		y\bigl((n+1)T\bigr) = \delta_2 y\bigl(nT+\bar{T}\bigr).
	\end{cases}
\end{equation}

%

By analyzing the Poincaré map associated with the switched system \eqref{AA}, Zheng et al. \cite{Zheng2025} established the following competitive outcomes between \textit{Aedes aegypti} and \textit{Aedes albopictus}:{
\begin{enumerate}[\rm (i)]
	\item If $\delta_2 \leq \delta_2^*$ and $\delta_1 \leq \delta_1^*$, then both species go extinct.
	\item If $\delta_2 \leq \delta_2^*$ and $\delta_1 > \delta_1^*$, then \textit{Aedes aegypti} wins the competition.
	\item If $\delta_2 > \delta_2^*$ and $\delta^* \leq 0$, then \textit{Aedes albopictus} wins the competition.	
\end{enumerate}}
Here,
\begin{equation}
{\delta_i^* = e^{-(\beta_i-\mu_i)\bar{T}}, \quad i=1,2, \quad \text{and} \quad \delta^* = \frac{1}{\xi_1}\ln \frac{\delta_1}{\delta_1^*} - \frac{1}{\xi_2}\ln \frac{\delta_2}{\delta_2^*}.}
\end{equation}

While $\delta_2 > \delta_2^*$ and $\delta^* > 0$ (we call it as case (iv) in this paper, see Lemma 2.3), Zheng et al. \cite[Theorem 3.6]{Zheng2025} proved that there exists a unique positive $T$-periodic solution. However, as pointed out by {Zheng et al.:}``the instability of the positive periodic solution has not been rigorously proven, although our numerical simulation shows its instability. Numerically, when $\delta_2 > \delta_2^*$ and $\delta^* > 0$, model \eqref{AA} generates bistable dynamics in that one population drives the other to extinction (see \cite[p. 1205]{Zheng2025})''. We note here that bistable dynamics means that such a unique coexistence periodic solution is of saddle type, and that a separatrix passing through this saddle coexistence periodic solution divides the basins of attraction of the two single-population periodic solutions (see, e.g., Smith and Thieme \cite{Smith2001} and references therein).


In the present paper, we will tackle this problem on the instability of the unique positive periodic solution; and moreover, we will rigorously
establish the bistable dynamics and a complete classification of global dynamics for system \eqref{AA}. More precsiely, we perform a detailed estimation of the Floquet multipliers for the positive periodic solution of system \eqref{AA} and prove that the solution is always unstable (see Lemma \ref{saddle}). By appealing to the theory of monotone dynamical systems, we establish the bistable nature of the dynamics (see Theorem \ref{thm31}). We then provide a complete classification of the global dynamics for system \eqref{AA} (see Corollary \ref{cor}), where the distribution diagram of the global dynamics in the $\delta_1$-$\delta_2$ plane is illustrated in Figure \ref{fig:para-region}. To examine the influence of switching parameters on competitive outcomes, we further analyze the bifurcations of the global dynamics as the switching duration $\bar{T}$ varies over the interval $(0,T]$.

%
%

This paper is organized as follows. Section 2 introduces the notation and provides relevant definitions and preliminary results. In Section 3, we analyze the Floquet multipliers of the positive periodic solution of system (\ref{AA}) and establish the bistable dynamics. We then derive the complete classification of the global dynamics for this system,
as well as  the parameter regions of four distinct dynamic
scenarios. Section 4 discusses in detail the influence of the switching parameters on the global dynamics.
The paper ends with a conclusion in Section 5.

\section{Notations and Preliminaries}\label{sec2}
 Let $\mathbb{R}^2_+:=\{x\in \mathbb{R}^2:x_i\geq 0,  i=1,2 \}$.
Let $X\subset \mathbb{R}^2_+$ and $P:X \rightarrow X$
be the Poincar$\acute{\rm e}$ (period) map of system (\ref{AA}).
The orbit of a state $x$ for $P$ is $\gamma(x)=\{P^n(x)$, $n\in
\mathbb{Z}_+ \}$. A fixed point $x$ of $P$ is a point $x\in X$ such that
$P(x)=x$. A set $V\subset X$ is positively
invariant under $P$ if $PV\subset V$, and invariant if $PV=V$. If $P$ is a differentiable map, we write $DP(x)$ as the Jacobian matrix of
$P$ at the point $x$. For an $n\times n$ matrix $A$, we write $A\geq
0$ iff $A$ is a nonnegative matrix (i.e., all the entries are
nonnegative) and $A\gg 0$ iff $A$ is a positive matrix (i.e., all the
entries are positive).

For two vectors $x,y \in \mathbb{R}^2$, we write $x\leq y$
if $x_i\leq y_i$ $i=1,2$, and $x\ll y$ if $x_i<y_i$ $i=1,2$.
If $x\leq y$ but $x\neq y$, we write $x<y.$
For $x,y \in \mathbb{R}_+^2$, let $[x,y]:=\{z\in \mathbb{R}_+^2:x\leq z \leq y\}$
be the closed order interval, and
$[[x,y]]:=\{z\in \mathbb{R}_+^2:x\ll z \ll y\}$ the open order interval.

Let $K:=\{(x_1,x_2)\in \mathbb{R}^2:x_1\geq 0, x_2\leq0\}$. Then we write $x\leq_K y$ if $y-x\in K$;
$x<_K y$ if $y-x\in K\setminus \{(0,0)\}$; and $x\ll_K y$ if $y-x\in $ Int$K$.
Notations such as $x\geq_K (>_K, \gg_K)\ y$ have the natural meanings.
Let $C_1:=\{(x_1,0):x_1 \in R_+\}$ and $C_2:=\{(0,x_2):x_2\in R_+\}$.
For any $a, b \in \mathbb{R}_+^2$ and $a\ll_K b$, we define the
closed $K$-order interval
 $[a,b]_K:=\{x\in \mathbb{R}_+^2: a\leq_K x \leq_K b\}$ and the open $K$-order internal
$[[a,b]]_K:=\{x\in \mathbb{R}_+^2: a\ll_K x \ll_K b\}$.


For any initial value $(x_0, y_0) \in \mathbb{R}^2_+ \setminus \{(0, 0)\}$, system \eqref{AA1} determines the state $(x(\bar{T}), y(\bar{T}))$ at the end of the favorable season. Starting from this state, the population sizes at the end of the full cycle are given by $(x(T), y(T)) = (\delta_1 x(\bar{T}), \delta_2 y(\bar{T}))$ according to the discrete mapping in system \eqref{AA2}. By iterating this procedure, the solution of system \eqref{AA} is uniquely determined and well-defined for all $t \in [0, +\infty)$. Since system \eqref{AA} is $T$-periodic, the associated Poincaré map $\mathcal{P}$ is defined as
$$\mathcal{P}(x(0), y(0)) = (x(T), y(T)).$$

We first define a linear map $L$ by
$$L(x,y) = (\delta_1 x, \delta_2 y), \quad (x,y) \in \mathbb{R}_+^2.$$
Let $\{\Phi_t\}_{t \geq 0}$ denote the solution flow associated with the autonomous system \eqref{AA1}. The Poincar\'e map $\mathcal{P}$ can then be expressed as the composition
$$\mathcal{P}(x,y) = (L \circ \Phi_{\bar{T}})(x,y), \quad \forall (x,y) \in \mathbb{R}_+^2 \setminus \{(0,0)\},$$
or more compactly, $\mathcal{P} = L \circ \Phi_{\bar{T}}$. Subsequently, we shall investigate the long-term dynamics of the discrete-time system defined by the iterations $\{\mathcal{P}^n\}_{n \geq 0}$.

\begin{lemma}\label{lemma:bounedness}
The set $\mathcal{M} = ([0, M_1] \times [0, M_2]) \setminus \{(0,0)\}$, where $M_i = \frac{\beta_i - \mu_i}{\xi_i} + 1$ for $i=1,2$, is positively invariant under the Poincar\'e map $\mathcal{P}$.\end{lemma}
\begin{proof}
For any $(x,y) \in \mathcal{M}$ such that $x = M_1$, we have$$\frac{dx}{dt} = M_1 \left( \frac{\beta_1 M_1}{M_1 + \alpha y} - \mu_1 - \xi_1 M_1 - \xi_1 y \right) \leq M_1 (\beta_1 - \mu_1 - \xi_1 M_1) = -M_1 \xi_1 < 0.$$Similarly, for any $(x,y) \in \mathcal{M}$ such that $y = M_2$, we have$$\frac{dy}{dt} = M_2 (\beta_2 - \mu_2 - \xi_2 x - \xi_2 M_2) \leq M_2 (\beta_2 - \mu_2 - \xi_2 M_2) = -M_2 \xi_2 < 0.$$These inequalities imply that the vector field points inward on the boundaries $x=M_1$ and $y=M_2$. Thus, $\Phi_t(\mathcal{M}) \subset \mathcal{M}$ for all $t > 0$. In particular, at $t = \bar{T}$, we have $\Phi_{\bar{T}}(\mathcal{M}) \subset \mathcal{M}$.

Since $L(x,y) = (\delta_1 x, \delta_2 y)$ with $\delta_i \leq 1$ for $i=1,2$, it follows that $L(\mathcal{M}) \subset \mathcal{M}$. Consequently,
$$\mathcal{P}(\mathcal{M}) = L \circ \Phi_{\bar{T}}(\mathcal{M}) \subset L(\mathcal{M}) \subset \mathcal{M}.$$
The proof is complete.
\end{proof}

Let $F(x,y) = (xf_1(x,y), yf_2(x,y))$ denote the right-hand side of system \eqref{AA1} with
$$f_1(x, y)=\frac{\beta_1 x}{x + \alpha y} - \mu_1 - \xi_1(x + y),\ f_2(x, y)=\beta_2-\mu_2 - \xi_2(x+y).$$
Integrating \eqref{AA1} yields
$$x(t) = x(0) \exp \left( \int_0^t f_1(x(s), y(s)) ds \right)$$
and
$$y(t) = y(0) \exp \left( \int_0^t f_2(x(s), y(s)) ds \right)$$
for any $(x(0), y(0)) \in \mathbb{R}_+^2 \setminus \{(0,0)\}$. {It follows from Lemma \ref{lemma:bounedness} that}
$$\Phi_t(x,y) = (x(t), y(t)) \to (0,0)$$ as $(x(0), y(0)) \to (0,0)$ for all $t \geq 0$. By defining $\Phi_t(0,0) = (0,0)$, we obtain a continuous extension of $\Phi_t$ onto $\mathbb{R}_+^2$. Consequently, the Poincar\'e map can be naturally extended by setting $\mathcal{P}(0,0) = (0,0)$.



\begin{lemma}
\label{monotone}
The map $\mathcal{P}$ has the following properties on $\mathbb{R}_+^2$.
\begin{enumerate}[\rm (i)]
    \item If $(x_1,y_1) \leq_K (x_2,y_2)$, then $\mathcal{P}(x_1,y_1) \leq_K \mathcal{P}(x_2,y_2)$. Moreover, if
    $(x_1,y_1) <_K (x_2,y_2)$ and $(x_2,y_2)\in ${\rm Int} $\mathbb{R}_+^2$, then $\mathcal{P}(x_1,y_1)\ll_K \mathcal{P}(x_2,y_2)$.
    \item For any $(x(0),y(0)) \in \mathbb{R}_+^2$, the sequence of points $\mathcal{P}^n(x(0),y(0))$ converges to a fixed point of $\mathcal{P}$ as $n \to \infty$.
\end{enumerate}
\end{lemma}

\begin{proof}
See Hsu and Zhao \cite[Lemma 2.2]{SX2012} or Niu et al. \cite[Lemma 3.1]{NWX2021}.
\end{proof}

\begin{lemma}
\label{local dynamics}
For system \eqref{AA}, the following properties hold:{
\begin{enumerate}[\rm (i)]
  \item If $\delta_2 \leq \delta_2^*$ and $\delta_1 \leq \delta_1^*$, then the trivial fixed point $O \coloneqq (0, 0)$ is globally asymptotically stable in $\mathbb{R}_+^2$.
  \item If $\delta_2 \leq \delta_2^*$ and $\delta_1 > \delta_1^*$, then the boundary fixed point $R_1=(\hat{x},0)$ on the $x$-axis is globally asymptotically stable in $\mathbb{R}_+^2\setminus C_2$.
  \item If $\delta_2 > \delta_2^*$ and $\delta^* \leq 0$, then the boundary fixed point $R_2 = (0, \hat{y})$ on the $y$-axis is globally asymptotically stable in $\mathbb{R}_+^2\setminus C_1$.
  \item If $\delta_2 > \delta_2^*$ and $\delta^* > 0$, then $\mathcal{P}$ has a unique positive fixed point $E = (x^*, y^*)$.
\end{enumerate}}
\end{lemma}
\begin{proof}

{The conclusions follow from Theorem 3.7 and Theorem 3.6 in Zheng et al. \cite{Zheng2025} immediately.}
\end{proof}

\begin{remark}
\label{rk24}
{\rm For case (iv) in Lemma \ref{local dynamics}, Zheng et al. \cite[p. 1205]{Zheng2025} proposed the question for the instability of the positive fixed point $E$; and moreover, they further numerically illustrated the  phenomenon of bistable dynamics in that one population drives the other to extinction.}

\end{remark}




\section{Bistable Dynamics and Complete Classification}\label{sec3}

{As mentioned in the introduction and Remark \ref{rk24}}, the bistable dynamics of system \eqref{AA} means that the positive fixed point $E$ is unstable, with a separatrix passing through $E$ dividing the basins of attraction of the boundary fixed points $R_1$ and $R_2$. We present the main results in this paper.

\begin{theorem}[Instability and bistable dynamics]
\label{thm31}
For $\delta_2 > \delta_2^*$ and $\delta^* > 0$,
\begin{enumerate}[\rm (a)]
\item  the unique positive fixed point $E$ is always a saddle (and hence, unstable);
\item system \eqref{AA} exhibits the bistable dynamics:
there exists a continuous, invariant, and unordered \textup{(}with respect to $\leq_K$\textup{)} curve $\Gamma \subset \mathbb{R}_+^2$, which is through the positive fixed point $E$ that separates the basins of attraction of $R_1$ and $R_2$.
\end{enumerate}
\end{theorem}

Theorem \ref{thm31} provides a theoretical justification for the numerical simulations on bistable dynamics reported by Zheng et al. in \cite[p. 1205]{Zheng2025}. Precisely, when $\delta_2 > \delta_2^*$ and $\delta^* > 0$ (i.e., case (iv) in Lemma \ref{local dynamics}), system \eqref{AA} admits a separatrix curve $\Gamma$ that connects the origin $O$ with the positive fixed point $E$ and extends to infinity in $\mathbb{R}_+^2$, such that {\it Aedes albopictus} wins whenever the initial distribution lies above $\Gamma$, whereas {\it Aedes aegypti} wins whenever it lies below $\Gamma$. Based on this result, we further establish a complete classification of the global dynamics for the system in the following corollary.

\begin{corollary}[Complete classification]
\label{cor}
For system \eqref{AA}, the global dynamics can be completely classified as follows:
\begin{enumerate}
\item [{\rm I}] {\rm (}\textit{Bistability}{\rm).} If $\delta_2 > \delta_2^*$ and $\delta^* > 0$, then both $R_1$ and $R_2$ are locally asymptotically stable, and $\mathcal{P}$ admits a unique positive fixed point $E$ of saddle type. Furthermore, there exists a continuous, unbounded, invariant, and unordered \textup{(}with respect to $\leq_K$\textup{)} curve $\Gamma \subset \mathbb{R}_+^2$ that separates the basins of attraction of $R_1$ and $R_2$.
{  \item[{\rm II}]  {\rm (}\textit{Aedes albopictus wins}{\rm).} If $\delta_2 > \delta_2^*$ and $\delta^* \leq 0$, then $\mathcal{P}$ has no positive fixed point, and $R_2$ is globally asymptotically stable in $\mathbb{R}_+^2 \setminus C_1$.}
{  \item[{\rm III}] {\rm (}\textit{Aedes aegypti wins}{\rm).} If $\delta_2 \leq \delta_2^*$ and $\delta_1 > \delta_1^*$, then $\mathcal{P}$ has no positive fixed point, and $R_1$ is globally asymptotically stable in $\mathbb{R}_+^2 \setminus C_2$.}
   \item[{\rm IV}]  {\rm (}\textit{Global Extinction}{\rm).} If $\delta_2 \leq \delta_2^*$ and $\delta_1 \leq \delta_1^*$, then $\mathcal{P}$ admits no boundary or positive fixed points, and the origin $O$ is globally asymptotically stable in $\mathbb{R}_+^2$.
\end{enumerate}
\end{corollary}
\begin{proof}
{It is a direct combination of our main result Theorem \ref{thm31} and Lemma \ref{local dynamics}(i)-(iii).}
 \end{proof}

In order to prove Theorem \ref{thm31}, we first investigate the Floquet multipliers and local stability of the positive fixed point of $\mathcal{P}$, as established in the following lemma.

\begin{lemma}[Instability of the positive periodic solution]
\label{saddle}
Let $E = (x^*, y^*)$ be a positive fixed point of $\mathcal{P}$, and let $\lambda_1, \lambda_2$ denote the eigenvalues of the Jacobian matrix $D\mathcal{P}(E)$. Then $0 < \lambda_1 < 1 < \lambda_2$, which implies that $E$ is a hyperbolic saddle.
\end{lemma}
\begin{proof}
 Let $F(x, y)$ denote the vector field on the right-hand side of system \eqref{AA1}, and let $(\bar{x}(t), \bar{y}(t)) = \Phi_t(E)$ be the solution flow starting at $E$. The variational equation is given by
 $$
 \frac{dV(t)}{dt} = DF(\bar{x}(t), \bar{y}(t)) V(t)
 $$
 with $V(0) = I$, where $V(t) = D_x\Phi_t(E)$.

Define the scaling matrix
$$
Q(t)= \begin{pmatrix} \bar{x}(t) & 0 \\ 0 & \bar{y}(t) \end{pmatrix}^{-1}.
$$
Under the transformation $w = Q(t)v$, the linear system
$$
\frac{dv}{dt} = DF(\bar{x}(t), \bar{y}(t)) v
$$
is transformed into
\begin{equation}\label{ww}
\frac{dw}{dt} = B(t)w
\end{equation}
where
$$
B(t) = \begin{pmatrix}
\cfrac{\alpha\beta_1  \bar{x}(t) \bar{y}(t)}{(\bar{x}(t) + \alpha \bar{y}(t))^2} - \xi_1 \bar{x}(t) & -\cfrac{\alpha \beta_1 \bar{x}(t) \bar{y}(t)}{(\bar{x}(t) + \alpha \bar{y}(t))^2} - \xi_1 \bar{y}(t) \\[3pt]
-\xi_2 \bar{x}(t) & -\xi_2 \bar{y}(t)
\end{pmatrix}.
$$
Let $W(t)$ be the monodromy matrix of system (\ref{ww}). That is, $W(t)$ satisfies
$$
\frac{dW(t)}{dt} = B(t) W(t), \quad W(0) = I.
$$
Thus, we have
$$
W(t) = Q(t) V(t) Q^{-1}(0),
$$
and hence,
$$
V(t) = Q^{-1}(t) W(t) Q(0), \quad \forall t \geq 0.
$$

At $t = \bar{T}$, using the condition $(\bar{x}(\bar{T}), \bar{y}(\bar{T})) = \left( \cfrac{x^*}{\delta_1}, \cfrac{y^*}{\delta_2} \right)$, we have
$$
D\mathcal{P}(E) = \begin{pmatrix} \delta_1 & 0 \\ 0 & \delta_2 \end{pmatrix} V(\bar{T}) = \begin{pmatrix} x^* & 0 \\ 0 & y^* \end{pmatrix} W(\bar{T}) \begin{pmatrix} x^* & 0 \\ 0 & y^* \end{pmatrix}^{-1},
$$
which implies that $D\mathcal{P}(E)$ is similar to $W(\bar{T})$.

Let
\begin{equation}\label{ss}
S(t)= \begin{pmatrix} 1 & 0 \\ 0 & -1 \end{pmatrix} W(t) \begin{pmatrix} 1 & 0 \\ 0 & -1 \end{pmatrix}^{-1}, \quad t \geq 0.
\end{equation}
Then $S(t)$ satisfies
$$
\frac{dS}{dt} = A(t) S, \quad S(0) = I
$$
with
$$
A(t) = (a_{ij}(t)) := \begin{pmatrix}
\cfrac{ \alpha\beta_1 \bar{x}(t) \bar{y}(t)}{(\bar{x}(t) + \alpha \bar{y}(t))^2} - \xi_1 \bar{x}(t) & \cfrac{\alpha\beta_1  \bar{x}(t) \bar{y}(t)}{(\bar{x}(t) + \alpha \bar{y}(t))^2} + \xi_1 \bar{y}(t) \\
\\
\xi_2 \bar{x}(t) & -\xi_2 \bar{y}(t)
\end{pmatrix}.
$$

Since the matrix $A(t)$ is cooperative and irreducible, it follows from Smith and Waltman \cite[Theorem B.3]{HP1995} that $S(t)\gg 0$ for each $t > 0$. Clearly, $W(\bar{T})$ is similar to $S(\bar{T})$, and hence
$$
r\left(D\mathcal{P}\left(E\right)\right) = r(W(\bar{T})) = r(S(\bar{T})).
$$
By the Perron–Frobenius theorem, $\lambda_2:= r(S(\bar{T}))$ is a simple eigenvalue of $S(\bar{T})$ with the maximum modulus and a positive eigenvector $u$.
Let $s(t)=(s_1(t),s_2(t))^\tau:=S(t)u$.
Then $$s(\bar{T})=S(\bar{T}) u = \lambda_2 u,$$
and $s_i(t) > 0$ for all $t \geq 0$, $i = 1, 2$. Since
$$
\begin{cases}
\cfrac{ds_1(t)}{dt} = \cfrac{\beta_1 \alpha \bar{x}(t) \bar{y}(t)}{(\bar{x}(t) + \alpha \bar{y}(t))^2}(s_1(t) + s_2(t)) - \xi_1 \bar{x}(t) s_1(t) + \xi_1 \bar{y}(t) s_2(t), \\
\cfrac{ds_2(t)}{dt} = \xi_2 \bar{x}(t) s_1(t) - \xi_2 \bar{y}(t) s_2(t),
\end{cases}
$$
it follows that
$$
\xi_2 \frac{ds_1(t)}{dt} + \xi_1 \frac{ds_2(t)}{dt} = \frac{\xi_2 \beta_1 \alpha \bar{x}(t) \bar{y}(t)}{(\bar{x}(t) + \alpha \bar{y}(t))^2}(s_1(t) + s_2(t)).
$$

By integrating the above equation for $t$ from $0$ to $\bar{T}$, we get
$$
(\xi_2 u_1 + \xi_1 u_2)(\lambda_2 - 1) = \int_0^{\bar{T}} \frac{\xi_2 \beta_1 \alpha \bar{x}(t) \bar{y}(t)}{(\bar{x}(t) + \alpha \bar{y}(t))^2}(s_1(t) + s_2(t)) dt > 0,
$$
which implies that $\lambda_2 > 1$.

Note that $V(t)$ satisfies the linear variational equation
\begin{equation}\label{vv}
\frac{dV(t)}{dt} = DF(\bar{x}(t), \bar{y}(t)) V(t), \quad V(0) = I.
\end{equation}
Consider the adjoint equation
\begin{equation}\label{yy}
\frac{dY(t)}{dt} = -DF(\bar{x}(t), \bar{y}(t))^{\tau} Y(t), \quad Y(0) = I,
\end{equation}
where $\tau$ denotes the transpose.
Then $Y(t) = (V(t)^\tau)^{-1}$. Since $-DF(\bar{x}(t), \bar{y}(t))^\tau$ is cooperative and irreducible (see Smith \cite[Chapter 4]{HL1995}, it follows from Smith
\cite[Theorem D(ii)]{Smith1986} that
$$
Y(t) = (V(t)^\tau)^{-1} \gg 0 \quad \text{for } t > 0.
$$
Therefore,
$$
(V(t))^{-1} \gg 0 \quad \text{for } t > 0.
$$


By a straightforward computation, it then follows that under a change of variable $z = Q(t)^{-1}v$, the linear equation
$$
\frac{dv}{dt} = -DF(\bar{x}(t), \bar{y}(t))^{\tau} v
$$
becomes
\begin{equation}\label{zz}
\frac{dz}{dt} = -B(t)^\tau z.
\end{equation}

Let $Z(t)$ be the monodromy matrix of this system, that is, $Z(t)$ satisfies
$$
\frac{dZ(t)}{dt} = -B(t)^\tau Z(t), \quad Z(0) = I.
$$
Thus, we have
$$
{Z(t) = Q(t)^{-1} Y(t) Q(0),\quad \forall t > 0.}
$$
Given that $Y(t) \gg 0$, $Z(t)$ is a positive matrix. Furthermore, the inverse of the Jacobian matrix $D\mathcal{P}(E)$ can be expressed as
$$D\mathcal{P}(E)^{-1} = V(\bar{T})^{-1} \begin{pmatrix} \delta_1 & 0 \\ 0 & \delta_2 \end{pmatrix}^{-1} = Q(0)^{-1} Z(\bar{T})^\tau Q(0).$$
Consequently, $D\mathcal{P}(E)^{-1}$ is similar to $Z(\bar{T})^\tau$, which implies that $D\mathcal{P}(E)^{-1}$ and $Z(\bar{T})$ share the same set of eigenvalues. By the Perron–Frobenius theorem, $\rho_2 := r(Z(\bar{T}))$ is a simple eigenvalue of $Z(\bar{T})$ with the maximum modulus and a positive eigenvector ${\hat{\beta}= (\hat{\beta}_1, \hat{\beta}_2)^\tau}$. Then $z(\bar{T})=Z(\bar{T}) {\hat{\beta}}= \rho_2 {\hat{\beta}}$, and $z_i(t) > 0$ for all $t \geq 0$, $i = 1, 2$. Since
$$
\begin{cases}
\cfrac{dz_1(t)}{dt} = \left(-\cfrac{\alpha {\hat{\beta}_1} \bar{x}(t) \bar{y}(t)}{(\bar{x}(t) + \alpha \bar{y}(t))^2} + \xi_1 \bar{x}(t)\right) z_1(t) + \xi_2 \bar{x}(t) z_2(t), \\
\cfrac{dz_2(t)}{dt} = \left(\cfrac{\alpha{\hat{\beta}_1}  \bar{x}(t) \bar{y}(t)}{(\bar{x}(t) + \alpha \bar{y}(t))^2} + \xi_1 \bar{y}(t)\right) z_1(t) + \xi_2 \bar{y}(t) z_2(t),
\end{cases}
$$
it follows that
$$
\frac{dz_1(t)}{dt} + \frac{dz_2(t)}{dt} = \left(\xi_1 z_1(t) + \xi_2 z_2(t)\right)(\bar{x}(t) + \bar{y}(t)).
$$

By integrating the above equation for $t$ from $0$ to $\bar{T}$, we find that this implies $\rho_2 > 1$. Therefore,
$$
\lambda_1 = \frac{1}{\rho_2} < 1
$$
is the smaller eigenvalue of $DP(E)$. Consequently, $0 < \lambda_1 < 1 < \lambda_2$, and $E$ is a hyperbolic saddle. We have proved the lemma.
\end{proof}


Now, we are ready to give the proof of Theorem \ref{thm31}.

\noindent \textit{\textbf{Proof of Theorem \ref{thm31}.}} Let $$U:=[R_2,R_1]_K=\{x\in \mathbb{R}_+^2: R_2\leq_K x \leq_K R_1\}.$$
Since $\delta_2>\delta_2^*$, it follows from {Lemma \ref{local dynamics}(ii)} that
$R_2$ is locally stable. By the expression of $\delta^*$, the conditions $\delta^*>0$ and $\delta_2>\delta_2^*$
implies that $\delta_1>\delta_1^*$. Due to { Lemma \ref{local dynamics}(iii)}, $R_1$ is also locally stable.
Using Proposition 2.1 in Hsu et al. \cite{SHP1996},
$\mathcal{P}$ has a positive fixed point $E$ in $U$, and $$R_2\ll_K E \ll_K R_1.$$
Moreover, Lemma \ref{saddle}
and {Lemma \ref{local dynamics}(iv)}
imply that $E$ is a unique saddle point.
\par Let $B_1$ and $B_2$ represent the basins of attraction of $R_1$ and $R_2$ for $\mathcal{P}$ in $\mathbb{R}_+^2$,
respectively. Then, $B_1$ and $B_2$ are relatively open in $\mathbb{R}_+^2$. Define $$\Gamma:=\mathbb{R}_+^2 \setminus (B_1\cup B_2).$$
Clearly, $O, E\in \Gamma$ and $\Gamma$ is \textit{invariant}. We first show that \textit{$\Gamma$ is unbounded}. If not, there exists a bounded open ball $B\subset \mathbb{R}^2$
with its closure $\bar{B}$ such that $\Gamma \subset \bar{B}$. Then, $$\mathbb{R}_+^2\setminus \bar{B}=(B_1\setminus \bar{B})\cup (B_2\setminus \bar{B}),$$ which contradicts the
connectedness of $\mathbb{R}_+^2\setminus \bar{B}$.
\par Next, we will prove that any two distinct points in $\Gamma$ are unordered related to $``\leq_K"$. If not,
there are two points $x,y \in \Gamma$ with $x<_K y$.
By using $\Gamma\setminus\{O\}\subset$ Int$\mathbb{R}_+^2$ and
Lemma \ref{monotone}(i), we have
$\mathcal{P}^nx\ll_K \mathcal{P}^ny$ for all $ n \geq 1$. Since $x,y \notin B_1\cup B_2$,
 it follows from
Lemma \ref{monotone}(ii) that
$$\lim\limits_{n\rightarrow\infty}
\mathcal{P}^n x=\lim\limits_{n\rightarrow\infty}
\mathcal{P}^n y=E.$$ This implies that $E$ attracts an open subset
$[[\mathcal{P}^{n_1}x,\mathcal{P}^{n_1} y]]_K$. By Lemma \ref{saddle},
$E$ is a saddle point, which can not attract the open set $[[\mathcal{P}^{n_1}x,\mathcal{P}^{n_1} y]]_K$.
This leads to a contradiction. Therefore, $\Gamma$ is unordered with repect to $``\leq_K"$.

Finally, we will show that $\Gamma$ is a continuous curve. Noticing that $\mathcal{P}$ is
strongly monotone (with respect to $``\leq_K"$) in Int$\mathbb{R}_+^2$.
Let $\tilde{\Gamma}=\Gamma \setminus \{O\}$. Then the pair $$\big((B_1\cup\Gamma)\cap {\rm Int}\mathbb{R}_+^2,
(B_2\cup\Gamma)\cap {\rm Int}\mathbb{R}_+^2\big)$$ is an order decomposition of Int$\mathbb{R}_+^2$
(see Tak$\rm\acute{a}\check{c}$ \cite[Definition 1.1]{T1992}) with $\tilde{\Gamma}$
being the boundary (also called \textit {d-hypersurface} of the order-decomposition).
Due to Tak$\rm\acute{a}\check{c}$ \cite[Proposition 1.3]{T1992}, it follows that $\tilde{\Gamma}$ is a one-dimensional
Lip-continuous graph. Precisely, we can write $\tilde{\Gamma}$ as a graph
$$\{(x_1,g(x_1))\in \mathbb{R}_+^2:x_1>0\},$$
where $g:(0,+\infty)\rightarrow\mathbb{R}_+$
is a strictly increasing Lip-continuous function. Let $$R=\lim\limits_{x_1\rightarrow 0^+}(x_1,g(x_1)).$$
Then, $R\in \partial \mathbb{R}_+^2\cap \Gamma$, which implies that $R=O$. Therefore,
$\Gamma=\tilde{\Gamma}\cup \{O\}$ and $\Gamma$ is a continuous curve. We have completed the proof.
\hfill$\square$

Before ending this section, we would like to discuss the dynamic scenarios and their distribution in the $\delta_1$-$\delta_2$ parameter plane.

\begin{remark}
   {\rm
Based on Corollary \ref{cor}, the parameter regions corresponding to the four distinct dynamic scenarios for system \eqref{AA} in the $\delta_1$-$\delta_2$ plane are explicitly given as follows:
\begin{itemize}
  \item[{I}] {\rm (Bistability): $\mathcal{D}_{\rm I} \coloneqq \left\{ (\delta_1, \delta_2) \in \mathbb{R}_+^2 : \delta_1^* < \delta_1 \leq 1, \ \delta_2^* < \delta_2 < h(\delta_1) \right\}$;}
{  \item[{II}] {\rm ({\it Aedes albopictus} wins): $\mathcal{D}_{\rm II} \coloneqq \left\{ (\delta_1, \delta_2) \in \mathbb{R}_+^2 : 0 < \delta_1 \leq h^{-1}(\delta_2), \ \delta_2^* < \delta_2 \leq 1 \right\}$;}}
 {   \item[{III}]   {\rm ({\it Aedes aegypti} wins): $\mathcal{D}_{\rm III} \coloneqq \left\{ (\delta_1, \delta_2) \in \mathbb{R}_+^2 : \delta_1^* < \delta_1 \leq 1, \ 0 < \delta_2 \leq \delta_2^* \right\}$;}}
  \item[{IV}] {\rm (Global Extinction): $\mathcal{D}_{\rm IV} \coloneqq \left\{ (\delta_1, \delta_2) \in \mathbb{R}_+^2 : 0 < \delta_1 \leq \delta_1^*, \ 0 < \delta_2 \leq \delta_2^* \right\}$.}
\end{itemize}
{\rm Here, $\delta_2 = h(\delta_1)$ denotes the separatrix curve of $\mathcal{D}_{\rm I}$ and $\mathcal{D}_{\rm II}$, with $h^{-1}(\delta_2)$ being its inverse function. The geometric configuration of these regions is illustrated in Figure \ref{fig:para-region}.}}
\end{remark}

\begin{figure}[htbp]
  \centering
  \begin{minipage}[b]{0.32\textwidth}
    \centering
    \includegraphics[width=\linewidth]{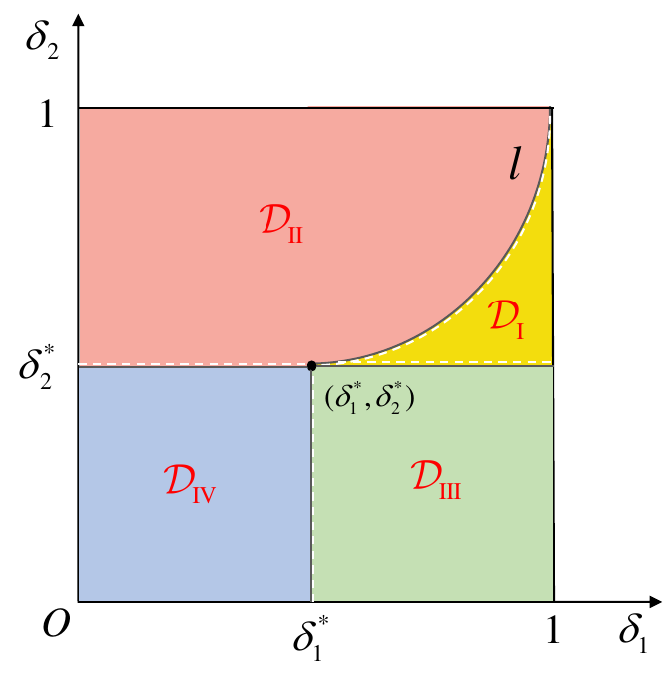}
    \subcaption{$\xi_2 > \xi_1$}
  \end{minipage}
  \hfill
  \begin{minipage}[b]{0.32\textwidth}
    \centering
    \includegraphics[width=\linewidth]{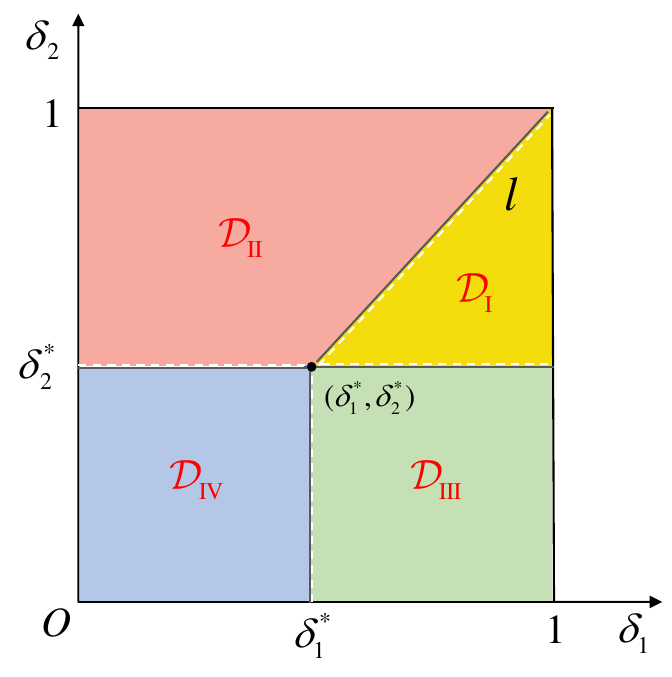}
    \subcaption{$\xi_2 = \xi_1$}
  \end{minipage}
  \hfill
  \begin{minipage}[b]{0.32\textwidth}
    \centering
    \includegraphics[width=\linewidth]{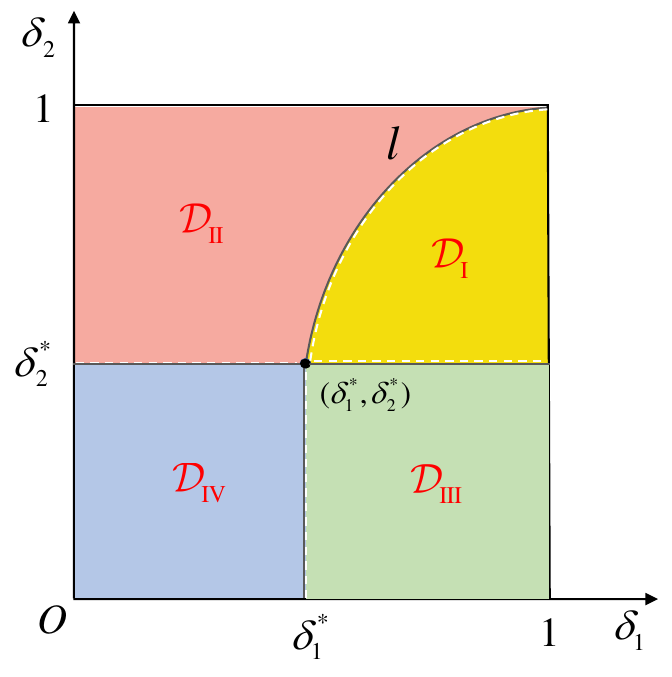}
    \subcaption{$\xi_2 < \xi_1$}
  \end{minipage}
  \caption{Distribution diagram of global dynamics for system \eqref{AA} in the $\delta_1$-$\delta_2$ plane under different relation between the density-dependent death rates $\xi_1$ and $\xi_2$.}
  \label{fig:para-region}
\end{figure}

To precisely determine the geometric shape of the boundary curve $l \coloneqq \{(\delta_1, h(\delta_1)) \in \mathbb{R}_+^2 : \delta_1^* \leq \delta_1 \leq 1\}$, we analyze the implicit function $\delta_2 = h(\delta_1)$ defined by the algebraic equation
\begin{equation}
  \frac{1}{\xi_1} \ln \frac{\delta_1}{\delta_1^*} - \frac{1}{\xi_2} \ln \frac{\delta_2}{\delta_2^*} = 0.
  \label{eq:separatrix}
\end{equation}
Direct differentiation yields
\[
  \frac{d\delta_2}{d\delta_1} = \frac{\xi_2 \delta_2}{\xi_1 \delta_1} > 0 \quad \text{and} \quad \frac{d^2\delta_2}{d\delta_1^2} = \frac{\xi_2 \delta_2}{(\xi_1 \delta_1)^2}(\xi_2 - \xi_1).
\]
Consequently, the curvature of the separatrix curve $l$ on the interval $[\delta_1^*, 1]$ depends strictly on the ratio of $\xi_1$ and $\xi_2$, where $\xi_1$ and $\xi_2$
represent the density-dependent death rates of {\it Aedes aegypti} and {\it Aedes albopictus}, respectively.

\begin{enumerate}[\rm (i)]
  \item If $\xi_2 > \xi_1$, the curve $\delta_2 = h(\delta_1)$ is strictly increasing and convex (see Figure \ref{fig:para-region}(a));
  \item If $\xi_2 = \xi_1$, the curve degenerates into a straight line segment with slope $\frac{\delta_2^*}{\delta_1^*}$ (see Figure \ref{fig:para-region}(b));
  \item If $\xi_2 < \xi_1$, the curve $\delta_2 = h(\delta_1)$ is strictly increasing and concave (see Figure \ref{fig:para-region}(c)).
\end{enumerate}

Furthermore, the separatrix curve passes through $(\delta_1^*, \delta_2^*)$, since $h(\delta_1^*) = \delta_2^*$. In particular, when $\frac{\beta_1-\mu_1}{\xi_1} = \frac{\beta_2-\mu_2}{\xi_2}$, it also passes through $(1, 1)$, i.e., $h(1) = 1$. If $\frac{\beta_1-\mu_1}{\xi_1} \neq \frac{\beta_2-\mu_2}{\xi_2}$, then $h(1) \neq 1$; nevertheless, the global bifurcation diagrams can be derived in an entirely analogous manner and are thus omitted here for brevity.

\begin{remark}
{\rm Biological insights and dynamical transitions across parameter boundaries depicted in Figure \ref{fig:para-region} can be summarized as follows:}
\begin{enumerate}[\rm (i)]
  \item {\rm{Varying $\delta_1$ (Control of species $x$):} {For a fixed $\delta_2 \leq \delta_2^*$, increasing $\delta_1$ forces the system to transition from global extinction ({Region $\mathcal{D}_{\rm IV}$}) to competitive exclusion of $y$ ({Region $\mathcal{D}_{\rm III}$}) as $(\delta_1, \delta_2)$ crosses the vertical line $\delta_1 = \delta_1^*$. Conversely, for $\delta_2 > \delta_2^*$, increasing $\delta_1$ drives the dynamics from competitive exclusion of $x$ ({Region $\mathcal{D}_{\rm II}$}) directly into bistability ({Region I}) upon crossing the separatrix $l$.}}
  \item {\rm{Varying $\delta_2$ (Control of species $y$):} {For a fixed $\delta_1 \leq \delta_1^*$, continuously increasing $\delta_2$ shifts the state from {Region $\mathcal{D}_{\rm IV}$} to {Region $\mathcal{D}_{\rm II}$} across $\delta_2 = \delta_2^*$. However, for $\delta_1 > \delta_1^*$, raising $\delta_2$ induces a rich succession of dynamical regimes: transitioning from {Region $\mathcal{D}_{\rm III}$} to {Region $\mathcal{D}_{\rm I}$}, and ultimately to {Region $\mathcal{D}_{\rm II}$} as $(\delta_1, \delta_2)$ successively crosses $\delta_2 = \delta_2^*$ and the separatrix $l$.}}
  \item {\rm Varying $\xi_2$: The shape and extent of {Region $\mathcal{D}_{\rm I}$}, together with the geometry of the separatrix $l$, depend on the relative
  magnitudes of $\xi_1$ and $\xi_2$. As $\xi_2$ decreases, {Region $\mathcal{D}_{\rm I}$} expands, and bistability can be realized over a broader range of parameters $(\delta_1,\delta_2)$. This implies that $\xi_2$ directly governs the parameter region of $(\delta_1,\delta_2)$ supporting bistability. The smaller
    $\xi_2$, the more readily bistable dynamics arise.}
\end{enumerate}
\end{remark}



\section{The effects of the switching parameter $\bar{T}$}\label{sec4}
{In this section, we fix all parameters except $\bar{T}\in (0,T]$, and further analyze the bifurcations of the global
dynamics of system (\ref{AA}) as $\bar{T}$ varies over $(0,T]$. For convenience, Corollary \ref{cor} can be summarized into the following four dynamic scenarios:
\begin{itemize}
\item[(i)] Dynamics-I: Bistability occurs for {\it Aedes aegypti} and {\it Aedes albopictus};
{\item[(ii)] Dynamics-II: {\it Aedes albopictus} wins the competition;}
{\item[(iii)] Dynamics-III: {\it Aedes aegypti} wins the competition;}
\item[(iv)] Dynamics-IV: {\it Aedes aegypti} and {\it Aedes albopictus} go to extinction.
\end{itemize}
Now, we discuss two cases $\bar{T}=T$ and $\bar{T}\in (0,T)$,
and further obtain the following two propositions.
\begin{proposition}
\label{prop1}
If $\bar{T}=T$, then system (\ref{AA}) admits only two dynamic scenarios, which are presented in Table \ref{table1}. This result also corresponds to Theorem 2.1 in \cite{Zheng2025}.
\end{proposition}
\begin{table}[htbp]
\centering
\caption{Parameter conditions and dynamic scenarios for $\bar{T}=T$}
\label{table1}
\renewcommand{\arraystretch}{2.8}
\begin{tabular}{
|>{\centering\arraybackslash}m{4.2cm}|
>{\centering\arraybackslash}m{4.8cm}|
>{\centering\arraybackslash}m{4.0cm}|
}
\hline
Values of $\bar{T}$ & Parameter conditions & Dynamic scenarios \\
\hline
\multirow{2}{*}{$\bar{T}= T$}
& $\displaystyle \frac{\beta_1-\mu_1}{\xi_1} > \frac{\beta_2-\mu_2}{\xi_2}$
& Dynamics-I \\
\cline{2-3}
& $\displaystyle \frac{\beta_1-\mu_1}{\xi_1} \leq \frac{\beta_2-\mu_2}{\xi_2}$
& Dynamics-II \\
\hline
\end{tabular}
\end{table}
\begin{proof}
In this case, there is no unfavorable season, which implies that system (\ref{AA}) becomes the autonomous system (\ref{AA1}) and satisfies
$\delta_1=\delta_2=1$. Thus, $\delta_1>\delta_1^*$ and $\delta_2>\delta_2^*$. By the expression of $\delta^*$, we have
$$\delta^*=\Big(\frac{\beta_1-\mu_1}{\xi_1}-\frac{\beta_2-\mu_2}{\xi_2}\Big)\bar{T}.$$
Then the conclusions follow from Corollary \ref{cor}(I) and Corollary \ref{cor}(II) immediately.
\end{proof}

\begin{proposition}
\label{prop2}
If $\bar{T}\in (0,T)$, then system (\ref{AA}) has four dynamic scenarios, and the bifurcations of the global dynamics are fully demonstrated in the Table \ref{table2} and Figure \ref{bifurcation-1}-\ref{bifurcation-3}, where
$$\bar{T}_i:=\frac{\ln \delta_i}{\mu_i-\beta_i},\ i=1,2,$$
and
$$\bar{T}^*=\bigg(\frac{\beta_2-\mu_2}{\xi_2}\bar{T}_2-\frac{\beta_1-\mu_1}{\xi_1}\bar{T}_1\bigg)\bigg/\bigg(\frac{\beta_2-\mu_2}{\xi_2}-\frac{\beta_1-\mu_1}{\xi_1}\bigg).$$

\end{proposition}
\begin{table}[htbp]
\centering
\caption{The bifurcations of the global dynamics for $\bar{T}\in (0,T)$}
\label{table2}
\begin{tblr}{
colspec={|m{1.85cm}|m{2.1cm}|m{1.8cm}|m{1.8cm}|m{1.8cm}|m{1.8cm}|m{1.8cm}|},
rows={valign=m,halign=c},
rowsep=6pt
}
\hline
\SetCell[r=2]{}{\small Relations of  $\bar{T}_2$ and  $\bar{T}_1$} & \SetCell[r=2]{}Parameter conditions & \SetCell[c=4]{}Regions of dynamic scenarios & & & & \SetCell[r=2]{}Bifurcation diagrams \\
\cline{3-6}
& &Dynamics-I &Dynamics-II &Dynamics-III &Dynamics-IV & \\
\hline
\SetCell[r=2]{}$\bar{T}_2=\bar{T}_1$ &{\tiny $\frac{\beta_1-\mu_1}{\xi_1} > \frac{\beta_2-\mu_2}{\xi_2}$} &$(\bar{T}_1,T)$ &{\scriptsize Not applicable} &{\scriptsize Not applicable} & $(0,\bar{T}_1]$  &Fig. 2(a) \\
\cline{2-7}
&{\tiny $\frac{\beta_1-\mu_1}{\xi_1} \leq \frac{\beta_2-\mu_2}{\xi_2}$} &{\scriptsize Not applicable} &$(\bar{T}_1,T)$ &{\scriptsize Not applicable} &$(0,\bar{T}_1]$ &Fig. 2(b) \\
\hline
\SetCell[r=2]{}$\bar{T}_2<\bar{T}_1$ &{\tiny $\frac{\beta_1-\mu_1}{\xi_1} > \frac{\beta_2-\mu_2}{\xi_2}$} &$(\bar{T}^*,T)$ &$(\bar{T}_2,\bar{T}^*]$  &{\scriptsize Not applicable} &$(0,\bar{T}_2]$ &Fig. 3(a) \\
\cline{2-7}
&{\tiny $\frac{\beta_1-\mu_1}{\xi_1} \leq \frac{\beta_2-\mu_2}{\xi_2}$} &{\scriptsize Not applicable} &$(\bar{T}_2,T)$ &{\scriptsize Not applicable} &$(0,\bar{T}_2]$ &Fig. 3(b) \\
\hline
\SetCell[r=2]{}$\bar{T}_2>\bar{T}_1$ &{\tiny $\frac{\beta_1-\mu_1}{\xi_1} \geq \frac{\beta_2-\mu_2}{\xi_2}$} &$(\bar{T}_2,T)$ &{\scriptsize Not applicable}  & $(\bar{T}_1,\bar{T}_2]$&$(0,\bar{T}_1]$ &Fig. 4(a) \\
\cline{2-7}
&{\tiny $\frac{\beta_1-\mu_1}{\xi_1} < \frac{\beta_2-\mu_2}{\xi_2}$} &$(\bar{T}_2,\bar{T}^*)$ &$[\bar{T}^*,T)$ &$(\bar{T}_1,\bar{T}_2]$ &$(0,\bar{T}_1]$ &Fig. 4(b) \\
\hline
\end{tblr}
\end{table}

\begin{proof}
To investigate the bifurcations of the global dynamics in this case,
we conduct a detailed analysis for three subcases according to the relationship between
$\bar{T}_1$ and $\bar{T}_2$.

\textbf{Subcase 1.} $\bar{T}_2=\bar{T}_1$
\begin{enumerate}[\rm (i)]
\item When $\bar{T}\in (0,\bar{T}_1]$, it is clear that $\bar{T}\leq \bar{T}_1$ and $\bar{T}\leq \bar{T}_2$, and then,
$\delta_1\leq \delta_1^*$ and $\delta_2\leq \delta_2^*.$ By Corollary \ref{cor}(IV), the dynamic scenario
for $\bar{T}\in (0,\bar{T}_1]$ is Dynamics-IV.
\item When $\bar{T}\in (\bar{T}_1,T)$, we have $\bar{T}> \bar{T}_1$ and $\bar{T}> \bar{T}_2$, and then,
$\delta_1>\delta_1^*$ and $\delta_2> \delta_2^*.$
Note that
$$\delta^*=\frac{1}{\xi_1}\ln\frac{\delta_1}{\delta_1^*}-\frac{1}{\xi_2}\ln\frac{\delta_2}{\delta_2^*}
=\bigg(\frac{\beta_2-\mu_2}{\xi_2}-\frac{\beta_1-\mu_1}{\xi_1}\bigg)(\bar{T}_1-\bar{T}).$$
Then, we have the following two conclusions.
\begin{itemize}
\item[(ii-1)] If {\footnotesize $\cfrac{\beta_1-\mu_1}{\xi_1} > \cfrac{\beta_2-\mu_2}{\xi_2}$}, then $\delta^*>0$.
By Corollary \ref{cor}(I), the dynamic scenario
for $\bar{T}\in (\bar{T}_1,T)$ is Dynamics-I (see Fig \ref{bifurcation-1}(a)).
\item[(ii-2)] If {\footnotesize $\cfrac{\beta_1-\mu_1}{\xi_1} \leq \cfrac{\beta_2-\mu_2}{\xi_2}$}, then $\delta^*\leq0$. By Corollary \ref{cor}(II),
 the dynamic scenario for $\bar{T}\in (\bar{T}_1,T)$ is Dynamics-II (see Fig \ref{bifurcation-1}(b)).

\end{itemize}
\par The corresponding bifurcation diagrams of this subcase are shown in Figure \ref{bifurcation-1}.
\end{enumerate}

\begin{figure}[htbp]
    \centering
    \begin{minipage}[b]{0.4\textwidth}
        \centering
        \includegraphics[width=\linewidth]{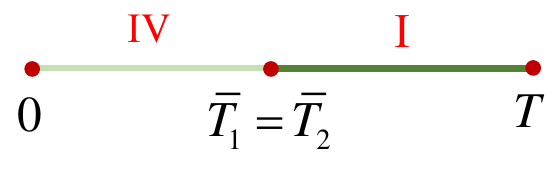}
        \subcaption{{\small $\frac{\beta_1-\mu_1}{\xi_1} > \frac{\beta_2-\mu_2}{\xi_2}$}}
    \end{minipage}
     \hfill
    \begin{minipage}[b]{0.4\textwidth}
        \centering
        \includegraphics[width=\linewidth]{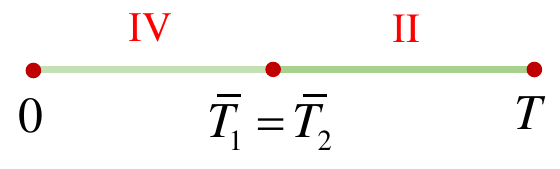}
        \subcaption{{\small $\frac{\beta_1-\mu_1}{\xi_1} \leq \frac{\beta_2-\mu_2}{\xi_2}$}}
    \end{minipage}
    \caption{$\bar{T}_2=\bar{T}_1$}
    \label{bifurcation-1}
\end{figure}

\textbf{Subcase 2.} $\bar{T}_2<\bar{T}_1$
\begin{enumerate}[\rm (i)]
\item When $\bar{T}\in (0,\bar{T}_2]$, we have $\bar{T}< \bar{T}_1$ and $\bar{T}\leq \bar{T}_2$, and then,
$\delta_1< \delta_1^*$ and $\delta_2\leq \delta_2^*.$ By Corollary \ref{cor}(IV), the dynamic scenario
for $\bar{T}\in (0,\bar{T}_2]$ is also Dynamics-IV.
\item When $\bar{T}\in (\bar{T}_2,\bar{T}_1]$, we have $\bar{T}\leq \bar{T}_1$ and $\bar{T}> \bar{T}_2$,
and then,
$\delta_1\leq \delta_1^*$ and $\delta_2> \delta_2^*.$ Meanwhile,
$$\delta^*=\frac{\beta_1-\mu_1}{\xi_1}(\bar{T}- \bar{T}_1)-\frac{\beta_2-\mu_2}{\xi_2}(\bar{T}- \bar{T}_2)<0.$$
By Corollary \ref{cor}(II), the dynamic scenario for $\bar{T}\in (\bar{T}_2,\bar{T}_1]$ is Dynamics-II.

\begin{figure}[htbp]
    \centering
    \begin{minipage}[b]{0.48\textwidth}
        \centering
        \includegraphics[width=\linewidth]{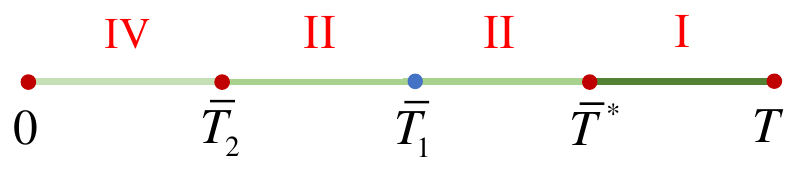}
        \subcaption{{\small $\frac{\beta_1-\mu_1}{\xi_1} > \frac{\beta_2-\mu_2}{\xi_2}$}}
    \end{minipage}
     \hfill
    \begin{minipage}[b]{0.4\textwidth}
        \centering
        \includegraphics[width=\linewidth]{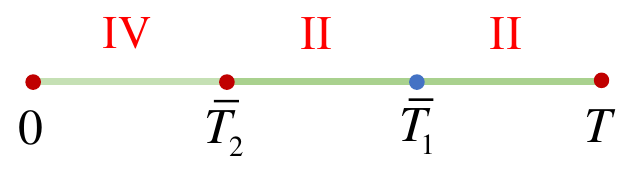}
        \subcaption{{\small $\frac{\beta_1-\mu_1}{\xi_1} \leq \frac{\beta_2-\mu_2}{\xi_2}$}}
    \end{minipage}
    \caption{$\bar{T}_2<\bar{T}_1$}
    \label{bifurcation-2}
\end{figure}

\item When $\bar{T}\in (\bar{T}_1,T)$, we have $\bar{T}> \bar{T}_1$ and $\bar{T}> \bar{T}_2$,
and then, $\delta_1> \delta_1^*$ and $\delta_2> \delta_2^*.$ Since
$$\delta^*=\frac{\beta_1-\mu_1}{\xi_1}(\bar{T}- \bar{T}_1)-\frac{\beta_2-\mu_2}{\xi_2}(\bar{T}- \bar{T}_2).$$
Then, we have the following two conclusions.
\begin{itemize}
\item[(iii-1)] If {\footnotesize $\frac{\beta_1-\mu_1}{\xi_1}\leq \frac{\beta_2-\mu_2}{\xi_2}$, then $\delta^*<0$}. By Corollary \ref{cor}(II),
the dynamic scenario for $\bar{T}\in (\bar{T}_1,T)$ is Dynamics-II (see Fig \ref{bifurcation-2}(b)).
\item[(iii-2)] If {\footnotesize $\frac{\beta_1-\mu_1}{\xi_1}> \frac{\beta_2-\mu_2}{\xi_2}$},
then $\delta^*$ can be either positive or negative. Thus, we denote
\begin{equation}
\label{T-value}
\bar{T}^*=\left(\frac{\beta_2-\mu_2}{\xi_2}\bar{T}_2-\frac{\beta_1-\mu_1}{\xi_1}\bar{T}_1\right)\big/\left(\frac{\beta_2-\mu_2}{\xi_2}-\frac{\beta_1-\mu_1}{\xi_1}\right).
\end{equation}
Recalling that $\bar{T}_1>\bar{T}_2$, it follows that $\bar{T}^*>\bar{T}_1$ and $\bar{T}^*>\bar{T}_2$. By simple calculations, we obtain that
$\delta^*\leq 0$ whenever $\bar{T} \in (\bar{T}_1,\bar{T}^*]$; while $\delta^*> 0$ whenever $\bar{T} \in (\bar{T}^*,T)$. Due to Corollary \ref{cor}(II) and (I), the dynamic scenario is Dynamics-II for $(\bar{T}_1,\bar{T}^*]$ and  is Dynamics-I
for $(\bar{T}^*,T)$ (see Fig \ref{bifurcation-2}(a)).

\end{itemize}
\par The corresponding bifurcation diagrams of this subcase are shown in Figure \ref{bifurcation-2}.
\end{enumerate}

\textbf{Subcase 3.} $\bar{T}_2>\bar{T}_1$
\begin{enumerate}[\rm (i)]
\item When $\bar{T}\in (0,\bar{T}_1]$, we have $\bar{T}\leq \bar{T}_1$ and $\bar{T}< \bar{T}_2$, and then,
$\delta_1\leq \delta_1^*$ and $\delta_2< \delta_2^*.$ By Corollary \ref{cor}(IV), the dynamic scenario
for $\bar{T}\in (0,\bar{T}_1]$ is also Dynamics-IV.
\item When $\bar{T}\in (\bar{T}_1,\bar{T}_2]$, we have $\bar{T}> \bar{T}_1$ and $\bar{T}\leq \bar{T}_2$,
and then,
$\delta_1> \delta_1^*$ and $\delta_2\leq \delta_2^*.$ By Corollary \ref{cor}(III), the dynamic scenario
for $\bar{T}\in (\bar{T}_1,\bar{T}_2]$ is Dynamics-III.

\begin{figure}[htbp]
    \centering
    \begin{minipage}[b]{0.4\textwidth}
        \centering
        \includegraphics[width=\linewidth]{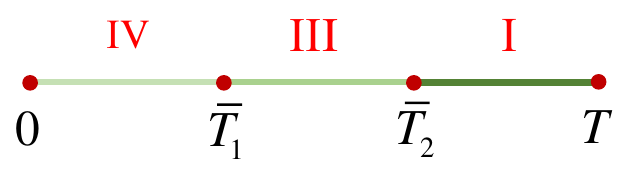}
        \subcaption{{\small $\frac{\beta_1-\mu_1}{\xi_1} \geq \frac{\beta_2-\mu_2}{\xi_2}$}}
    \end{minipage}
     \hfill
    \begin{minipage}[b]{0.5\textwidth}
        \centering
        \includegraphics[width=\linewidth]{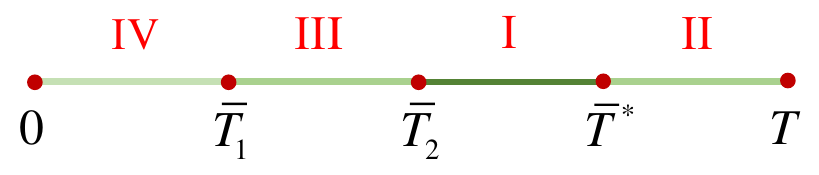}
        \subcaption{{\small $\frac{\beta_1-\mu_1}{\xi_1} < \frac{\beta_2-\mu_2}{\xi_2}$}}
    \end{minipage}
    \caption{$\bar{T}_2>\bar{T}_1$}
    \label{bifurcation-3}
\end{figure}
\item When $\bar{T}\in (\bar{T}_2,T)$, we have $\bar{T}> \bar{T}_1$ and $\bar{T}> \bar{T}_2$,
and then, $\delta_1> \delta_1^*$ and $\delta_2> \delta_2^*.$ Recalling that
$$\delta^*=\frac{\beta_1-\mu_1}{\xi_1}(\bar{T}- \bar{T}_1)-\frac{\beta_2-\mu_2}{\xi_2}(\bar{T}- \bar{T}_2)
,$$ we also have the following two conclusions.
\begin{itemize}
\item[(iii-1)] If {\footnotesize $\cfrac{\beta_1-\mu_1}{\xi_1}\geq \cfrac{\beta_2-\mu_2}{\xi_2}$, then $\delta^*>0$}. By Corollary \ref{cor}(I),
the dynamic scenario for $\bar{T}\in (\bar{T}_2,T)$ is Dynamics-I (see Fig \ref{bifurcation-3}(a)).
\item[(iii-2)] If {\footnotesize $\cfrac{\beta_1-\mu_1}{\xi_1}< \cfrac{\beta_2-\mu_2}{\xi_2}$},
then $\delta^*$ can be either positive or negative. Similar to Subcase 2(iii-2), we can define $\bar{T}^*$ (see the expression of (\ref{T-value})).
Since $\bar{T}_2>\bar{T}_1$, it follows that $\bar{T}^*>\bar{T}_1$ and $\bar{T}^*>\bar{T}_2$. A simple calculation yields that
$\delta^*>0$ whenever $\bar{T} \in (\bar{T}_2,\bar{T}^*)$; while $\delta^*\leq 0$ whenever $\bar{T} \in [\bar{T}^*,T)$. From Corollary \ref{cor}(I) and (II), the dynamic scenario is Dynamics-I
for $(\bar{T}_2,\bar{T}^*)$ and Dynamics-II
for $[\bar{T}^*,T)$ (see Fig \ref{bifurcation-3}(b)).
\end{itemize}
The corresponding bifurcation diagrams of this subcase are shown in Figure \ref{bifurcation-3}.
\end{enumerate}
Consequently, the proof of this proposition is now complete.
\end{proof}

\begin{remark}
{\rm For $\bar{T}\in (0,T)$, $\bar{T}^*$ does exist and $\bar{T}^*=\bar{T}_1=\bar{T}_2$ whenever $\bar{T}_2=\bar{T}_1$.
However, $\bar{T}^*$ does not exist for the two cases $\bar{T}_2<\bar{T}_1$, $\frac{\beta_1-\mu_1}{\xi_1}\leq \frac{\beta_2-\mu_2}{\xi_2}$
and $\bar{T}_2>\bar{T}_1$, $\frac{\beta_1-\mu_1}{\xi_1}\geq \frac{\beta_2-\mu_2}{\xi_2}$ (see Figures \ref{bifurcation-2}(b) and \ref{bifurcation-3}(a)).
Biologically, $\bar{T}_i$ can be interpreted as the critical minimal growing-season length for the persistence of the $i$-th species. When the favorable season length in every cycle exceeds this threshold, the mosquito species can proliferate in the growing season and persist across cycles via overwintering survival; otherwise, the population tends to go extinct.}
\end{remark}

\section{Conclusion}

In this paper, we investigate an {\it Aedes aegypti} and {\it Aedes albopictus} competition model with favorable-unfavorable
seasonal shifts. We first obtain the stability of the positive fixed point via the stability analysis of equilibria.
By using the theory of monotone systems, we obtain the bistable dynamics for two mosquito species. That is,
there exists a separatrix $\Gamma$ connecting $O$ and $E$ to the
infinity in $\mathbb{R}_+^2$ such that {\it Aedes albopictus} wins whenever the initial distribution
is above $\Gamma$, while {\it Aedes aegypti} wins whenever the initial distribution is
below $\Gamma$. Based on this, we give a complete classification of global dynamics of system (\ref{AA}).
In addition, we analyze the effect of the switching parameters on the competition outcomes.
More precisely, we provide the parameter regions of the global dynamics in the $\delta_1$-$\delta_2$ plane.
By taking $\bar{T}$ as a switching strategy parameter, the bifurcations of competitive outcomes are
studied in three subcases. The corresponding bifurcation diagrams are presented in Figures \ref{bifurcation-1}-\ref{bifurcation-3}.
These results state that this model does not exhibit the dynamic result of competitive coexistence,
instead, it either exhibits global extinction, competitive exclusion, or bistable dynamics under the influence of seasonal shifts.
Meanwhile, our conclusions also well explain the results of numerical analysis results in Zheng et al. \cite{Zheng2025}.


\begin{thebibliography}{99}


\bibitem{Bhatt2013}
S. Bhatt, P. W. Gething, O. J. Brady, J. P. Messina, A. W. Farlow, C. L. Moyes, J. M. Drake, J. S. Brownstein, A. G. Hoen, O. Sankoh, M. F. Myers, D. B. George, T. Jaenisch, G. R. W. Wint, C. P. Simmons, T. W. Scott, J. J. Farrar  and S. I. Hay, The Global Distribution and Burden of Dengue, {\em Nature}, \textbf{496}(2013), 504–507.

\bibitem{Barbara}
B. Byttebier, V. Loetti, M. S. D. Majo and S. Fischer, Temporal dynamics of the egg bank of {\it Aedes aegypti} (Diptera: Culicidae) in
the winter-spring transition in a temperate region, {\em Acta Trop.}, \textbf{255}(2024), 107227.

\bibitem{Denlinger}
D. L. Denlinger and P. A. Armbruster, Mosquito diapause, {\em Annu. Rev. Entomol.}, \textbf{59}(2014), 73–93.

\bibitem{Diagne}
C. Diagne, B. Leroy, A. C. Vaissière, R. E. Gozlan, D. Roiz, I. Jarić, J. M. Salles, C. J. A. Bradshaw and F. Courchamp, High and rising economic costs of biological invasions worldwide, {\em Nature}, \textbf{592}(2021), 571–576.

\bibitem{Sylvia}
S. Fischer, M. S. D. Majo, C. M. D. Battista, P. Montini, V. Loetti and R. E. Campos, Adaptation to temperate climates: Evidence of photoperiod-induced embryonic dormancy in {\it Aedes aegypti} in South America, {\em J. Insect Physiol.}, \textbf{117}(2019), 103887.

\bibitem{Girard2020}
M. Girard, C. B. Nelson, V. Picot and D. J. Gubler, Arboviruses: A Global Public Health Threat, {\em Vaccine}, \textbf{38}(2020), 3989–3994.

\bibitem{Gabriel2023}
Z. L. Gabriel, M. P. Alexander, F. A. O. Janeide, P. B. Brian, B. P. David and L. Yvonne-Marie, Global Distribution of {\it Aedes aegypti} and {\it Aedes albopictus} in a
Climate Change Scenario of Regional Rivalry, {\em Insects}, \textbf{14}(2023), 14010049.

\bibitem{SHP1996}
S. B. Hsu, H. L. Smith and P. Waltman, Competitive exclusion and coexistence for competitive systems on ordered Banach spaces, {\em Trans. Amer. Math. Soc.} 348
(1996), 4083-4094.

\bibitem{SX2012}
S. B. Hsu and X. Q. Zhao, A Lotka-Volterra competition model with seasonal succession, {\em J. Math. Biol.,} \textbf {64} (2012), 109-130.


\bibitem{Iwamura}
T. Iwamura, A. Guzman-Holst and K. A. Murray,  Accelerating Invasion Potential of Disease Vector {\it Aedes aegypti} under Climate
Change,  {\em Nat. Commun.}, \textbf{11}(2020), 2130.

\bibitem{Koch1974}
A. L. Koch, Coexistence Resulting from an Alternation of
Density Dependent and Density Independent Growth, {\em J. Theor. Biol.} \textbf{44}(1974), 373-386.


\bibitem{K2010}
C. A. Klausmeier, Successional state dynamics: A novel approach to modeling nonequilibrium foodweb dynamics, {\em J.
Theor. Biol.,} \textbf {262}(2010), 584-595.


\bibitem{Khormi}
H. M. Khormi and L. Kumar, Climate Change and the Potential Global Distribution of {\it Aedes aegypti}: Spatial Modelling Using GIS and
CLIMEX, {\em Geospat. Health}, \textbf{8}(2014), 405–415.

\bibitem{lwande2020}
O. W. Lwande, V. Obanda, A. Lindström, C. Ahlm, M. Evander, J. Näslund, and G. Bucht, Globe-Trotting {\it Aedes aegypti} and {\it Aedes albopictus}: Risk Factors for Arbovirus Pandemics, {\em Vector Borne Zoonotic Dis.}, \textbf{20}(2020), 71–81.


\bibitem{Mayer}
S. V. Mayer , R. B. Tesh and N. Vasilakis, The emergence of arthropod-borne viral diseases: A global prospective
on dengue, chikungunya and zika fevers, {\em Acta Trop.}, \textbf{166}(2017), 155–63.

\bibitem{Nasci1989}
R. Nasci, S. Hare and F. Willis, Interspecific mating between Louisiana strains of {\it Aedes
albopictus} and {\it Aedes aegypti} in the field and laboratory, {\em J. Am. Mosq. Control Assoc.}, \textbf{5}(1989), 416-421.

\bibitem{NWX2021}
L. Niu, Y. Wang and X. Xie, Carrying simplex in the Lotka-Volterra competition model with seasonal succession with applications, {\em Discrete Contin. Dyn. Syst. -B}, \textbf{26(4)}(2021), 2161–2172.

\bibitem{NWX2025}
L. Niu, Y. Wang and X. Xie, The dynamics of Lotka-Volterra competition models with seasonal succession:
I. Classification of 3-dimensional models, {\em SIAM J. Appl. Math.}, \textbf{85}(2025), 499-523.

\bibitem{NWX2026}
\ L. Niu, Y. Wang and X. Xie, Global dynamics of three-dimensional Lotka-Volterra competition models with seasonal succession: II. Uniqueness
of positive fixed points, {\em SIAM J. Appl. Math.,} {\textbf {86}}(2026), 2126-2147.


\bibitem{Naslund2021}
J. Näslund, C. Ahlm, K. Islam,  M. Evander, G. Bucht and O. W. Lwande,  Emerging Mosquito-Borne Viruses Linked to {\it Aedes aegypti}
and {\it Aedes albopictus}: Global Status and Preventive Strategies, {\em Vector Borne Zoonotic Dis.}, \textbf{21}(2021), 731–746.

\bibitem{Parker2019}
C. Parker, D. Ramirez and C. R. Connelly, State-wide survey of {\it Aedes aegypti} and {\it Aedes albopictus} (diptera: culicidae)
in Florida, {\em J. Vector Biol.}, \textbf{44}(2019), 210-215.

\bibitem{Reynolds1980}
C. S. Reynolds, Phytoplankton assemblages and their periodicity in stratifying
lake systems, {\em Holarct. Ecol.}, \textbf{3}(1980), 141-159.



\bibitem{SUZLD1986}
U. Sommer, Z. M. Gliwicz, W. Lampert and A. Duncan, The PEG-model of
seasonal succession of planktonic events in fresh waters,
 {\em Archiv f$\ddot{u}$r Hydrobiologie} \textbf{106}(1986), 433-471.


\bibitem{Smith1986}
H. L. Smith, Periodic solutions of periodic competitive and cooperative systems, {\em SIAM J. Math. Anal.}, \textbf(17)(1986), 1289-1318.



\bibitem{HP1995}
H. L. Smith and P. Waltman, The Theory of the Chemostat, {\rm Cambridge University Press} (1995).

\bibitem{HL1995}
H. L. Smith, Monotone Dynamical Systems. An Introduction to the Theory of Competitive
and Cooperative Systems, {\rm \ Math.\ Surveys Monogr.\ 41,\ AMS,\ Providence, RI} (1995).

\bibitem{Smith2001}
H. L. Smith and H. R. Thieme, Stable Coexistence and Bi-stability for Competitive Systems on Ordered Banach Spaces,
{\em J. Differential Equations}, {\textbf{176}}(2001), 195-222.

\bibitem{T1992}
P. Tak\'{a}\v{c}, Domains of attraction of generic $\omega$-limit sets for
strongly monotone discrete-time semigroups, {\em J. Reine Angew. Math.} 423 (1992), 101-173.

 \bibitem{Tang2025}
B. Tang, Y. Xiao and J. Wu, Dynamic behaviors of a periodic system with threshold policy-guided periodic and intermittent therapy of tumor, {\em SIAM J. Appl. Math.,} {\textbf {85}}(2025), 366-392.

\bibitem{Ruan2026}
C. Tian, L. Zou and S. Ruan, A two-species competition model on a compact metric graph for the invasion and competition of \textit{Aedes Aegypti} and \textit{Aedes Albopictus}
mosquitoes in Florida, {\em J. Math. Biol.,} \textbf{93}(2026), 8.

\bibitem{Wilkerson}
R. C. Wilkerson, Y. -M. Linton and D. Strickman, Mosquitoes of the World, Johns Hopkins University Press: Baltimore, MD, USA, 2020.

\bibitem {White2020}
E. R. White, A. Hastings, Seasonality in ecology: Progress and prospects in theory,
{\em Ecological Complexity}, \textbf{44}(2020), 100867. 

\bibitem{Xiao2026} J. Xiao, J. Yu and B. Zheng, Dynamics of a switched mosquito population model with periodic supplementary releases, {\em Chaos Solitons Fract.},  \textbf{212}(2026), 119028. 

\bibitem{Yang}
B. Yang, B. A. Borgert, B. W. Alto, C. K. Boohene, J. Brew, K. Deutsch, J. T. DeValerio, et al., Modelling distributions of {\it Aedes aegypti} and {\it Aedes albopictus} using climate, host density and interspecies competition, {\em PLoS Negl. Trop. Dis.}, \textbf{15}(2020), e0009063.

\bibitem{Zheng2025}
 B. Zheng, L. Chang, H. Guo and J. Yu, Global competitive dynamics of {\it Aedes aegypti} and {\it Aedes albopictus}
 with favorable-unfavorable seasonal shifts, {\em SIAM J. Appl. Math.}, \textbf{85}(2025), 1189-1211.






























\end{thebibliography}
\end{document}